\documentclass[a4paper,10pt]{amsart}
\usepackage[ansinew]{inputenc}
\usepackage{a4wide}
\usepackage{amsfonts}
\usepackage{amsmath}
\usepackage{amssymb}
\usepackage{amsthm}

\newtheorem{proposition}{Proposition}
\newtheorem {lemma}{Lemma}
\newtheorem {theorem}{Theorem}
\newtheorem {Corollary}{Corollary}
\newtheorem{remark}[theorem]{Remark}

\numberwithin{equation}{section}
\numberwithin{equation}{section}
\numberwithin{equation}{section}
\begin{document}
\title{On the spectral gap of Brownian motion with jump boundary}
\author{Martin Kolb}
\address{Department of Statistics, 1 South Parks Road, Oxford, OX1 3TG United Kingdom}
\email{kolb@stats.ox.ac.uk}
\author{Achim W\"ubker}
\address{Fachbereich Mathematik, Universit\"at Osnabr\"uck, Albrechtstrasse 28a, 49076 Osnabr\"uck, Germany}
\email{awuebker@Uni-Osnabrueck.de}
\thanks{A.W. is supported by Deutsche Forschungsgemeinschaft}

\begin{abstract}
In this paper we consider the Brownian motion with jump boundary and present a new proof of a recent result of Li, Leung and Rakesh concerning the exact convergence rate in the one-dimensional case. Our methods are different and mainly probabilistic relying on coupling methods adapted to the special situation under investigation. Moreover we answer a question raised by Ben-Ari and Pinsky concerning the dependence of the spectral gap on the jump distribution in a multi-dimensional setting.
\end{abstract}

\maketitle

\section{Introduction and Notation}
This article investigates the so called Brownian motion with jump boundary (BMJB) and related models, which in recent years gave rise to several interesting results (see e.g. \cite{GK1}, \cite{GK2}, \cite{K}, \cite{BaP1} and \cite{BaP2}). The process itself can be easily described. Consider a
Brownian motion (BM) with initial value $x_0$ in an open domain $D \subset \mathbb{R}^d$ which we assume to have a $C^{2,\alpha}$-boundary for convenience. When hitting the boundary $\partial D$ of $D$, the BM gets redistributed in $D$ according to the jump distribution $\nu$, runs again until it hits the boundary, gets redistributed and repeats this behavior forever. Obviously there are a number of possible generalizations, e.g. the BM can be replaced by a more general diffusion and the jump distribution might depend on the boundary point, which is hit by the diffusion (see \cite{BaP1}). In this work we restrict ourselves to the model of a diffusion -- corresponding to a 'good' elliptic differential operator $L$ -- with a fixed jump distribution and partly even to BM with a fixed jump distribution, a setting already leading to interesting and highly non-trivial questions concerning e.g. the relation between the spectral gap of the process, the spectrum of the Dirichlet Laplacian and the jump distribution. 

The \textit{first} main result of this paper establishes in a general setting the continuous dependence of the spectral gap of the process on the jump distribution, where continuity is meant with respect to the weak topology. This answers a question posed by I. Ben-Ari and R. Pinsky in \cite{BaP1}.

The \textit{second} main contribution of this work is a direct probabilistic route to the calculation of the spectral gap $\gamma_1(\nu)$ in the case of a one-dimensional Brownian motion in $D=(a,b)$ with an arbitrary jump distribution based on the famous probabilistic coupling method. This enables to understand the somewhat surprisingly fast rate to equilibrium and its independence on $\nu$ on a 'path level'. As was already observed in previous contributions to the theme of the present work it turns out that $\gamma_1(\nu)$ is strictly larger than the lowest eigenvalue $\lambda_0^{(a,b)}$ of $-\frac{1}{2}\frac{d^2}{dx^2}$ with Dirichlet boundary conditions (see e.g. \cite{LLR} in the general case and \cite{GK1}, \cite{GK2}, \cite{BaP1} and \cite{BaP2} in the case when $\nu$ belongs to the class of Dirac-measures). 
Even more remarkably the spectral gap $\gamma_1(\nu)$ is constant in $\nu$ and coincides with the second Dirichlet eigenvalue $\lambda_1^{(a,b)}$. This has already been known due to a result in \cite{BaP2} in combination with a theorem in \cite{LLR}. More precisely, Ben-Ari and Pinsky \cite{BaP2} had determined the real-valued subset of the spectrum of the generator. Hereupon, Li, Leung and Rakesh \cite{LLR} established that the spectrum is in fact a subset of the real line. Their proof is quite involved and uses highly non-trivial results from Fourier analysis. 
In \cite{LLR} it is explicitly mentioned that the authors can not provide an intuitive explanation for the $\nu$-independence of the spectral gap. This remark together with the open question concerning the dependence of the spectral gap on the jump distribution -- formulated by Ben-Ari and Pinsky in \cite{BaP2} -- initiated the present work.

 Despite our purely theoretical study we want to mention that the process considered in this work or closely related versions are used in financal applications and also as simple models of processes in neuroscience and operations research. Moreover we point out that some known results together with (often slightly different) proofs are included in this work in order to be self-contained and readable independent from previous works.

This work is organized as follows: 
In section 2, we analyze the $L$-diffusion with jump boundary from a probabilistic point of view. More precisely, we re-prove already-known results concerning exponential ergodicity using classical renewal theory, which yields non-tangible upper bounds for the convergence rate. A sharp lower bound for this rate is derived in section 3.   
In section 4 we answer the question posed by Iddo Ben-Ari and Ross Pinsky by establishing the continuity of the spectral gap $\gamma_1(\nu)$ in $\nu$ with respect to the weak topology in a general multi-dimensional setting. This continuity property is then used in order to reduce the explicit calculation of $\gamma_1(\nu)$ in the case of a one-dimensional BMJB to jump-distributions $\nu$ with compact support in $(a,b)$. Finally, in section 5 we construct an efficient coupling in the case of a one-dimensional BMJB with compactly supported jump distribution $\nu$. This coupling is used in order to determine $\gamma_1(\nu)$.

In the remainder of this introduction we fix the notation and recall some basic facts which will be needed throughout this work. 
\subsection{Notation}
We introduce the process in higher dimensions and later specialize to the one dimensional situation. Let $D \subset \mathbb{R}^d$ be a domain with $C^{2,\alpha}$-boundary ($1>\alpha>0$) and let $(\Omega, (B_t)_{t \geq 0},(\mathbb{P}_x)_{x \in D})$ denote a smooth uniformly elliptic diffusion in $D$ which is killed after hitting the boundary of $D$, i.e. $(B_t)_{t \geq 0}$ is a diffusion process associated to a generator $L$ of the form 
\begin{displaymath}
L:= \frac{1}{2}\sum_{i,j=1}^da_{ij}(x)\partial_{ij} + \sum_{i=1}^d b_i(x)\partial_i,
\end{displaymath}
where the matrix $a=(a_{ij})_{i,j=1}^d$ is uniformly elliptic with symmetric coefficients $a_{ij}$, i.e. $a_{ij}=a_{ji}$. Moreover, we assume that
$a_{ij}$ and $b_i$ are bounded and have bounded derivatives. Of course, somewhat weaker assumptions are possible, but we do not aim for such a generality in order to avoid technical difficulties.
This process induces a compact semigroup of bounded operators $(P^D_t)_{t \geq 0}$ in $L^{2}(D)$, which is generated by $L$,  with Dirichlet boundary conditions as operator acting in $L^2(D)$. The spectrum $\Sigma(-L)$ thus consists of a sequence $(\lambda_k^D)_{k=0}^{\infty}$ converging to infinity. If the operator $L$ is \textit{symmetric} then there exists an associated orthonormal basis (of a suitable $L^2$-space) $(\varphi_k^D)_{k=0}^{\infty}$ of eigenfunctions. For example $L = \frac{1}{2}\sum_{i,j=1}^d\partial_i(a_{ij}(x)\partial_{j})$ is symmetric in $L^2(D)$. We denote by $g^D(\cdot, \cdot)$, $p^D(t,\cdot,\cdot)$ the Green- and the transition-function associated to BM in $D$ killed at the boundary $\partial D$. If $\nu$ is a probability measure in $D$ we denote by $p^D(t,\nu,y)$ and $g^{D}(\nu,y)$ the expressions
\begin{displaymath}
\int_D p^D(t,x,y)\,\nu(dx)\,\text{   and   }\, \int_D g^D(x,y)\,\nu(dx),
\end{displaymath}
respectively. Observe that $(P^{D}_t)_{t \geq 0}$ acts also in a consistent way on $L^p(D)$ for every $1 \leq p \leq \infty$. 

Let $W^{\rho,1}$ be a $L$-diffusion in $D$  with initial distribution $\rho$ which is killed at $\partial D$. Moreover, let $(W^{\nu,i})_{i \geq 2}$ denote an independent family of killed $L$-diffusion in $D$ with initial distribution $\nu$, which is independent of $W^{\rho,1}$. Set $T^{\rho,\nu}_1=\inf \bigl\lbrace t \geq 0 \mid W^{\rho,1}_t \in \partial D \bigr\rbrace$, $S^{\nu}_i = \inf \bigl \lbrace t \geq 0 \mid W^{\nu,i}_t \in \partial D \bigr\rbrace $ and inductively we define $T^{\rho,\nu}_{i+1}:= T^{\rho,\nu}_i+S^{\nu}_{i+1}$ for $i \geq 1$.
The process $(X^{\rho,\nu}_t)$, called a diffusion with jump boundary starting from the initial distribution $\rho$, is now defined as
\begin{equation}\label{definitionBMJB}
 X^{\rho,\nu}_t := \mathbf{1}_{\lbrace 0 \leq t < T^{\rho,\nu}_1 \rbrace}W^{\rho,1}_t+\sum_{i=2}^{\infty}\mathbf{1}_{\lbrace T_{i-1}^{\rho,\nu}\leq t < T_{i}^{\rho,\nu}\rbrace}W^{\nu,i}_{t-T^{\rho,\nu}_{i-1}}.
\end{equation}
If $L$ is just the Laplacian then we simply write BMJB instead of Brownian motion with jump boundary. If the initial distribution is clear from the context we just write $(X^{\nu}_t)_{t \geq 0}$ instead of $(X^{\rho, \nu}_t)_{t \geq 0}$. 

The process $X_{t}^{\nu}$ induces an operator semi-group on $L^{\infty}(D)$ by
\begin{equation}
\mathcal{P}_{t}^{\nu}f(x)=\mathbb{E}_x\bigl[f(X^{\nu}_t)\bigr]=\int_{a}^{b}f(y)\mathbb{P}\bigl(X_{t}^\nu\in dy|X_{0}^{\nu}=x\bigr).
\end{equation}
Instead of $\mathbb{P}(X^{\rho,\nu}\in \cdot)$ we will often write $\mathbb{P}_{\rho}(X^{\nu}_t \in \cdot)$ or even $\mathbb{P}_{\rho}(X_t \in \cdot)$. 

Moreover, $\|\cdot\|_{TV}$ denotes the total variation norm and for a measure $\gamma$ and a measurable function $f$ we write occasionally $\langle \gamma, f \rangle$ instead of $\int f \, d\gamma$.

It has been shown by previous authors (\cite{GK1}, \cite{GK2}, \cite{GK3} and \cite{BaP2}) that this process is uniformly geometrically ergodic, i.e. there exists a probability measure $\mu^{\nu}$ such that 
\begin{equation}\label{gamma}
 -\lim_{t \rightarrow \infty}\frac{1}{t}\log\sup_ {x \in D}\bigl\| \mathbb{P}_x\bigl(X^{\nu}_t \in \cdot)-\mu^{\nu}(\cdot)\bigr\|_{TV} =  \gamma_1(\nu) >  0,
\end{equation}
It is well-known that this is in fact equivalent to 
\begin{equation}
\lim_{t\rightarrow\infty}-\frac{1}{t}\log\sup_{f\in L^{\infty},\|f\|_{\infty}=1}\bigl\|\mathcal{P}_{t}^{\nu}f-\int f \,d\mu^{\nu}\bigr\|_{\infty}=\gamma_1(\nu),
\end{equation}
and hence $\gamma_1(\nu)$ is referred to as the spectral gap.
Since the spectrum $\Sigma$ of an operator and its dual coincide, it is reasonable to investigate the spectral properties of $\mathcal{P}^{\nu}_t$ ($t>0$) by studying the spectrum of $\mathcal{S}^{\nu}_t$ on the Banach-space $L^1(D)$, which is often easier to analyze and defined by
\begin{equation}\label{dualsemigroup}
L^1(D) \ni g \mapsto \mathcal{S}^{\nu}_tg(\cdot):= \int_Dg(x)p^{\nu}(t,x,\cdot)\,dx,
\end{equation}
where $p^{\nu}(t,\cdot,\cdot)$ denotes the transition density of the process $(X^{\nu}_t)_{t \geq 0}$. The existence of the continuous transition density $p^{\nu}(t,\cdot,\cdot)$ can be deduced from the existence of the continuous transition density $p^D(t,\cdot,\cdot)$ via decomposing $\mathbb{E}_x[f(X^{\nu}_t)]$ according the number of jumps, which occured before time $t$. This is formulated in a precise way in equation (1.2) of \cite{BaP2}. In particular this implies that $\mathcal{P}_t^{\nu}$ maps bounded countinuous function onto bounded continuous functions.

It can be easily verified that
\begin{equation}
\langle \mathcal{S}^{\nu}_t g, f\rangle =\langle g, \mathcal{P}^{\nu}_t f\rangle\,\,\,\forall g\in L^{1}(D),f\in L^{\infty}(D),
\end{equation}
i.e.
\begin{equation}\label{duality}
\mathcal{S}_t^{\nu\ast}=\mathcal{P}^{\nu}_t.
\end{equation}
It is well-known from general spectral theory that this implies 
\[\Sigma(\mathcal{S}_t^{\nu})=\overline{\Sigma(\mathcal{P}^{\nu}_t)} \mbox{ for all }t\in\mathbb{R}_{+},\]
i.e. up to complex conjugation both operators have the same spectrum at a fixed time $t$. And due to compactness of $\mathcal{S}_t^{\nu}$ as shown in great generality in \cite{BaP1} one can apply the spectral mapping theorem in order to relate the spectrum of the semigroup to the spectrum of its generator (for a very readable account on semigroup theory we refer to \cite{EN}).

\section{Geometric Ergodicity using Renewal Theory}
In this section we give a short proof of ergodicity and geometric ergodicity of multi-dimensional $L$-diffusion with jump boundary. We want to stress that this result is well-known and not at all surprising. We include this result here, since our arguments slightly differ from the ones used in previous works and are rather short. The only basic ingredient is renewal theory (see in particular \cite{As} and \cite{T}). \\
We will need the following certainly known Lemma.
\begin{lemma}\label{filem}
Assume that we have for every $x\in D$ that $\mathbb{P}_{x}\bigl(X^{\nu}_t\in \cdot\bigr)\longrightarrow \mu^{\nu}$ in the weak topology. Then $\mu^{\nu}$ is 
the unique invariant measure for $X^{\nu}$, i.e.
\begin{equation}
\mu^{\nu}(A)=\mathbb{P}_{\mu^{\nu}}(X^{\nu}_t\in A) \mbox{ for all }A\in\mathcal{B}(D), \,t\in\mathbb{R}_{+}.
\end{equation}
\end{lemma}
\begin{proof}
The weak convergence $\mathbb{P}_{x}\bigl(X^{\nu}_t\in \cdot\bigr)\longrightarrow \mu^{\nu}$ implies that for all bounded and continuous $f\in C(D)$ 
\begin{eqnarray}
\int_D f\,d\mu&=&\lim_{s\rightarrow\infty}\mathbb{E}_{x}\bigl[f(X^{\nu}_{t+s})\bigr]=\lim_{s\rightarrow\infty}\mathbb{E}_{x}\bigl[\mathbb{E}_{x}\bigl[f(X^{\nu}_{t+s})|X^{\nu}_s\bigr]\bigr]\nonumber\\
&=&\lim_{s\rightarrow\infty}\mathbb{E}_{x}\bigl[\mathbb{E}_{X^{\nu}_s}f(X^{\nu}_t)\bigr]=\mathbb{E}_{\mu^{\nu}}\bigl[f(X^{\nu}_t)\bigr],
\end{eqnarray}
where the last equality follows from the continuity of $\mathbb{E}_{y}\bigl[f(X^{\nu}_t)\bigr]=\int_D p^{\nu}(t,y,z)f(z)\,dz$ with respect to $y \in D$. One way to establish such this continuity is via equation \eqref{homren0}.
\end{proof}
Now let us provide a direct route for deriving the form of the invariant measure.

The above Lemma suggests to calculate $\mathbb{E}_x\bigl[f(X^{\nu}_t)\bigr]$ for continuous functions $f\in C(D)$.  
\begin{equation}\label{homren0}
\begin{split}
  \mathbb{E}_x\bigl[f(X^{\nu}_t)\bigr] &= \mathbb{E}_x\bigl[f(X^{\nu}_t),T^{\nu}_1>t\bigr] + \mathbb{E}_x\bigl[f(X^{\nu}_t),\,T^{\nu}_1\le t\bigr] \\
  &= \mathbb{E}_x\bigl[f(X^{\nu}_t),T^{\nu}_1>t\bigr]  + \int_{0}^{t}\mathbb{E}_{x}\bigl[f(X^{\nu}_t)|T^{\nu}_1=s\bigr]\mathbb{P}_{x}\bigl(T^{\nu}_{1}\in\,ds\bigr) \\
  &=\mathbb{E}_x\bigl[f(X^{\nu}_t),T^{\nu}_1>t\bigr]+ \int_{0}^{t}\mathbb{E}_{\nu}\bigl[f(X^{\nu}_{t-s})\bigr]\mathbb{P}_{x}\bigl(T^{\nu}_{1}\in\,ds\bigr).
 \end{split}
\end{equation}
This is called a delayed renewal equation, which can be further analyzed by considering the associated pure renewal equation
\begin{equation}
\begin{split}
\mathbb{E}_\nu\bigl[f(X^{\nu}_t)\bigr]= \mathbb{E}_\nu\bigl[f(X^{\nu}_t),T^{\nu}_1>t\bigr] + \int_{0}^{t}\mathbb{E}_{\nu}\bigl[f(X^{\nu}_{t-s})\bigr]\mathbb{P}_{\nu}\bigl(T^{\nu}_{1}\in\,ds\bigr),
\end{split}
\end{equation} 
or for short
\[Z(t)=z(t)+Z*F(t),\]
with $Z(t)=\mathbb{E}_\nu\bigl[f(X^{\nu}_t)\bigr]$, $z(t)= \mathbb{E}_\nu\bigl[f(X^{\nu}_t),T^{\nu}_1>t\bigr]$ and $F(t)=\mathbb{P}_{\nu}\bigl(T^{\nu}_{1}\in[0,t]\bigr)$. It is well-known from general renewal theory (see e.g. \cite{As}) that
\[Z(t)=U*z(t),\]
where $U(t)=\sum_{n=0}^{\infty}F^{\ast n}(t)$ and $F^{\ast 0}$ is defined as the Dirac distribution function degenerated at zero.
Since $z(t)$ is directly Riemann integrable (for a definition see e.g. \cite{As}), we can apply the Key Renewal Theorem (see e.g. \cite{As}, p. 155),
which says that
\begin{equation}\label{homren-1}
Z(t)=U\ast z(t)\rightarrow\frac{1}{m}\int_{0}^{\infty}z(s)ds,\,\quad\,m=\int_{0}^{\infty}t\,dF(t)
\end{equation}
to obtain
\begin{equation}\label{homren1}
\begin{split}
\lim_{t\rightarrow\infty}\mathbb{E}_\nu\bigl[f(X^{\nu}_t)\bigr]&=\frac{1}{\mathbb{E}_{\nu}[T_{1}^{\nu}]}\int_{0}^{\infty}\mathbb{E}_\nu\bigl[f(X^{\nu}_s),T^{\nu}_1>s\bigr]ds\\
&=\frac{1}{\mathbb{E}_{\nu}[T_{1}^{\nu}]}\int_{D}\int_{0}^{\infty}p^{D}(t,\nu,y)\,dt f(y)dy\\
&=\int_{D}f(y)\frac{g^{D}(\nu,y)}{\mathbb{E}_{\nu}[T_{1}^{\nu}]}dy
\end{split}.
\end{equation}
Inserting (\ref{homren1}) into (\ref{homren0}) yields
\[
\lim_{t\rightarrow\infty}\mathbb{E}_x\bigl[f(X^{\nu}_t)\bigr]=\int_{D}f(y)\frac{g^{D}(\nu,y)}{\mathbb{E}_{\nu}[T_{1}^{\nu}]}dy
\]
We have just derived the form of the invariant measure in a constructive and probabilistically rather direct way. Let us capture the preceding considerations by 
\begin{theorem}\label{first}
Let $(X_t^{\nu})_{t \geq 0}$ denote a $L$-diffusion process in the domain $D$ with an arbitrary jump distribution $\nu$. Then the process $(X^{\nu}_t)_{t \geq 0}$ is ergodic in the sense that for all continuous functions $f$ we have
\begin{equation}
\lim_{t\rightarrow\infty}\mathbb{E}_x\bigl[f(X^{\nu}_t)\bigr]=\int_Df(y)\mu^{\nu}(dy),
\end{equation}
where the invariant distribution $\mu^{\nu}$ is given by 
\begin{displaymath}
 \mu^{\nu}(dy) = \frac{1}{m^{\nu}}g^{D}(\nu,y)\,dy,
\end{displaymath}
with 
\begin{equation}\label{drift}
 m^{\nu} = \mathbb{E}_{\nu}\bigl[T^{\nu}_1\bigr] = \int_D g^{D}(\nu,y)\,dy.
\end{equation} 
\end{theorem}
It is natural to ask whether it is possible to obtain more information from renewal theory. At this, the speed of convergence to the invariant measure is of main interest. If the convergence is geometrically fast with respect to the total-variation norm, the Markov process is called 
geometrically ergodic. For general renewal processes, two features of the inter-arrival distribution $F$ are necessary and 
sufficient for geometric ergodicity: 
Firstly, $F$ must be a spread out distribution (for a definition see \cite{As}), since otherwise even convergence in total variation fails to be true. Secondly, $F$ must have exponentially decreasing tails (see e.g. \cite{As}). In our setting the inter-arrival distribution $\mathbb{P}_{\nu}\bigl(T_{1}^{\nu}\in\cdot\bigr)$ share both properties, since $\mathbb{P}_{\nu}\bigl(T_{1}^{\nu}\in\cdot\bigr)$ even admits a density.  
 
Hence [2,Theorem 2.10] allows to refine \eqref{homren-1} to the assertion
\begin{equation}\label{geomergodicity}
Z(t)=U\ast z (t)=\frac{1}{m}\int_{0}^{\infty}z(s)ds + O(e^{-\epsilon t})
\end{equation}  
for some sufficiently small $\epsilon$. In this situation, the function $f$ in the definition of $z(t)$ is only required to be bounded and measurable.
Therefore, equation (\ref{geomergodicity}) implies geometric ergodicity of the Markov process $X_t$. Actually, the tails of $\mathbb{P}_{x}\bigl(T_{1}^{\nu}\in\cdot\bigr)$ are uniform-exponentially decreasing for $x\in(a,b)$, and hence uniform ergodicity can be deduced.\\
Despite these affirmative results, it is impossible to obtain good or even sharp convergence rates from general renewal theory directly, since the above $\epsilon$ has been deduced from a general coupling argument which does not fit to our particular situation. \\
Similar calculations to the ones presented above have already been used in previous works of Grigorescu and Kang \cite{GK1} as well as Ben-Ari and Pinsky \cite{BaP2} in order to show ergodicity of the process. Grigorescu and Kang \cite{GK1} consider the Laplace transform of an expression similar to (\ref{homren0}) and calculate the associated poles. As one expects from general operator theoretic principles the density of the equilibrium measure can be characterized as residuum of the pole at 0. Ben-Ari and Pinsky use a different but still quite analytic approach. A probabilistic approach using the Doeblin condition in a general setting is presented by Grigorescu and Kang in \cite{GK3}.

\section{An upper bound for the rate of convergence}
The next natural question consists of calculating the exponential convergence rate. Let us look at the one-dimensional case. In order to find the exponential rate of convergence the analyst usually works on the generator level and considers the following eigenvalue problem
\begin{equation}\label{evproblem}
-\frac{1}{2}u'' = \lambda u\,\text{  with 'boundary' condition  }\, u(a)=\int_a^b u(y)\,\nu(dy)=u(b).
\end{equation}
The exponential rate of convergence $\gamma_1(\nu)$ coincides with $\min\lbrace \Re \lambda \not = 0 \mid \exists\, u_{\lambda}\not=0 \text{  satisfying  }\eqref{evproblem}\rbrace$. Observe that \eqref{evproblem} includes the non-local boundary condition, which represents the jump mechanism of the process. Instead of investigating the stationary problem \eqref{evproblem}, we work directly on the path level. In this section we first prove a simple upper bound on the convergence rate in the case of the one-dimensional BMJB in an interval $D=(a,b)$. For this reason we introduce the reflection map $R$ defined by 
\begin{displaymath}
 R:(a,b) \rightarrow (a,b),\,\text{   }\, R(x):= a+b-x.
\end{displaymath}
Furthermore recall from section 1.1 that $(\lambda_k^{(a,b)})_{k=0}^{\infty}$ denotes the sequence of eigenvalues of the operator $-\frac{1}{2}\frac{d^2}{dx^2}$ in $(a,b)$ with Dirichlet boundary conditions at both endpoints.
\begin{proposition}\label{bound}
Let $d=1$ and let $(X_t)_{t\geq 0}$ denote the BMJB in $(a,b)$ with jump distribution $\nu$. Then we have 
\begin{displaymath}
\gamma_1 (\nu) \leq \lambda_1^{(a,b)}.
\end{displaymath}
\end{proposition}
\begin{proof}
Due to the trivial inequality 
\begin{equation}\label{lowerbound1}
 \bigl\|\mathbb{P}_x(X_t \in \cdot\,)-\mathbb{P}_{R(x)}(X_t \in \cdot\,) \bigr\|_{TV} \leq 2 \sup_{x \in (a,b)}\bigl\|\mathbb{P}_x(X_t\in \cdot\,)- \mu\bigr\|_{TV}
\end{equation}
it is enough to estimate the left side of \eqref{lowerbound1} from below. Observe first that by symmetry we have
\begin{equation}\label{symmetry1}
 \mathbb{P}_x(T^{\nu}_1 \in \cdot ) = \mathbb{P}_{R(x)}(T^{\nu}_1 \in \cdot).
\end{equation}
Using \eqref{symmetry1} in the last step we obtain
\begin{equation*}
\begin{split}
 \mathbb{P}_{x}(X_t \in A)-\mathbb{P}_{R(x)}(X_t \in A) &= \mathbb{P}_x(X_t \in A;\,T_1^{\nu} > t)-\mathbb{P}_{R(x)}(X_t \in A;\,T_1^{\nu} > t) \\
 &+ \int_0^t\mathbb{P}_{\nu}(X_{t-r} \in A)\mathbb{P}_x(T^{\nu}_1 \in dr) -  \int_0^t\mathbb{P}_{\nu}(X_{t-s} \in A)\mathbb{P}_{R(x)}(T^{\nu}_1 \in ds) \\
 &= \mathbb{P}_x(B_t \in A;\,T_{(a,b)} > t)-\mathbb{P}_{R(x)}(B_t \in A;\,T_{(a,b)} > t), 
\end{split}
\end{equation*}
where $(B_t)_{t \geq 0}$ is a one-dimensional Brownian motion and $T_{(a,b)}=\inf\lbrace t > 0\mid B_t \notin (a,b)\rbrace$. From the spectral decomposition of the killed BM in $(a,b)$, $\varphi^{(a,b)}_{2i}(x)=\varphi^{(a,b)}_{2i}(R(x))$ and $\varphi^{(a,b)}_{2i+1}(x)=-\varphi^{(a,b)}_{2i+1}(R(x))$ for all $i\in\mathbb{N}\cup\{0\}$ we deduce
\begin{equation*}
\begin{split}
\mathbb{P}_{x}(X_t \in A)&-\mathbb{P}_{R(x)}(X_t \in A) = \mathbb{P}_x(B_t \in A;\,T_{(a,b)} > t)-\mathbb{P}_{R(x)}(B_t \in A;\,T_{(a,b)} > t)\\
&= \sum_{i=0}^{\infty}e^{-\lambda_i^{(a,b)}t}\varphi^{(a,b)}_i(x)\int_A\varphi^{(a,b)}_i(y)\,dy -\sum_{i=0}^{\infty}e^{-\lambda_i^{(a,b)}t}\varphi^{(a,b)}_i(R(x))\int_A\varphi^{(a,b)}_i(y)\,dy \\
&= 2\sum_{i=0}^{\infty}e^{-\lambda_{2i+1}^{(a,b)}t}\varphi^{(a,b)}_{2i+1}(x)\int_A\varphi^{(a,b)}_{2i+1}(y)\,dy.
\end{split}
\end{equation*}
This together with \eqref{lowerbound1} clearly proves the assertion of the proposition.
\end{proof}
One should keep in mind the important role of the symmetry $(x,R(x))$ used heavily in Proposition \eqref{bound}. This symmetry will occur also later in the proof of the matching upper bound.

Thus the straightforward arguments used in this section demonstrate in a rather direct way that in the case of a one-dimensional BMJB with jump distribution $\nu$ 
\begin{equation}\label{lambda01}
\gamma_1(\nu) \leq \lambda_1^{(a,b)}.
\end{equation}
\\
Before we start with a more precise investigation of the spectral gap for general jump distributions we ask whether the upper bound in \eqref{lambda01} is sharp at least in a very special situation.
We have seen that the $L$-diffusion with jump distribution $\nu$ is ergodic and the unique invariant distribution $\mu^{\nu}$ has the form
\begin{displaymath}\label{maybesharp}
 \mu^{\nu}(dy) = \frac{1}{m^{\nu}}\int_Dg^{D}(x,y)\,\nu(dx)\,dy,
\end{displaymath}
where $m^{\nu} = \int_D\int_Dg^{D}(x,y)\,\nu(dx)\,dy$. 
It seems reasonable that the speed of convergence to equilibrium should be fast whenever jump- and invariant distribution coincide. This immediately suggest the question for the existence of a jump measure such that $\nu = \mu^{\nu}$. For simplicity let $L$ now be a symmetric operator, for which the theory of quasi-limiting distributions is more transparent. Still the result also holds up to obvious modifications in the non-symmetric case, but the corresponding results on the quasilimiting behavior are scattered widely in the literature. It turns out that there is a suitable jump-measure, namely the \textit{quasistationary distribution} for the killed $L$-diffusion in $D$, which is given by $\rho_D:=\varphi_0^D(y)\,dy$ (normalized to give total mass $1$). The notation quasi-stationary refers to the fact that for measurable $A\subset D$ 
\begin{displaymath}
 \mathbb{P}_{\rho_D}\bigl(W_t \in A \mid T_D > t\bigr) = \int_A\varphi^D_0(y)\,dy =\rho_D(A),
\end{displaymath}
where $(W_t)_{t \geq 0}$ denotes an $L$-diffusion and $T_D=\inf \lbrace t >0 \mid W_t \in \partial D\rbrace$. This fact has already been noticed by Ben-Ari and Pinsky in Theorem 1 of \cite{BaP1}, where it has also been shown that the quasi-stationary distribution defines the unique fixed point of the equation $\mu^{\nu} = \nu$.

Hence if (\ref{lambda01}) would be tight, it should be at least tight for the BMJB in $D$ with jump distribution $\nu=\rho_D$. That
this is actually the case is the content of the following
\begin{proposition}[\cite{BaP1}, Theorem 1]
Let $L$ be a symmetric operator in $D$ and let the jump distribution $\nu$ be given by the quasi-stationary distribution $\rho_D(dy)=\varphi_0^{D}(y)\,dy$. Then we have 
\begin{displaymath}
 \gamma_1(\rho_D) = \lambda_1^D.
\end{displaymath}
\end{proposition}
\begin{proof}
The proof differs from the one given in \cite{BaP1} and seems to be more probabilistic to us. Let $(X_t)_{t\geq 0}$ denote the L-diffusion in $D$ with jump distribution $\nu(dy) = \varphi_{0}(y)\,dy$.
For measurable $A$ we decompose $\mathbb{P}_x(X_t \in  A)$ according to whether there has been a jump before time $t$:
\begin{equation*}
\begin{split}
 \mathbb{P}_x(X_t \in  A) &= \mathbb{P}_x(X_t \in A, T_1^{\nu} \leq t) +  \mathbb{P}_x(X_t \in A, T_1^{\nu} > t) \\
&= \int_0^t\mathbb{P}_{\nu}\bigl(X_{t-s}\in A\big)\mathbb{P}_x(T_1^{\nu} \in ds) +\mathbb{P}_x(X_t \in A \mid T_1^{\nu} > t) \mathbb{P}_x(T_1^{\nu} > t) \\
&= \nu(A)-\nu(A)\mathbb{P}_x(T_1^{\nu} \geq t) + \mathbb{P}_x(X_t \in A \mid T_1^{\nu} > t) \mathbb{P}_x(T_1^{\nu} > t) \\
&= \nu(A)+ \bigl(\mathbb{P}_x(X_t \in A \mid T_1^{\nu} > t)-\nu(A)\bigr) \mathbb{P}_x(T_1^{\nu} > t).
\end{split}
\end{equation*}
Thus observing  that $-\lim_{t \rightarrow \infty}\frac{1}{t}\log\| \mathbb{P}_x(X^{\nu}_t \in \cdot \mid T_1^{\nu} > t)-\nu(\cdot)\|_{TV}= \lambda_1^D-\lambda_0^D$ (see e.g. \cite{P} and \cite{GQZ}) we arrive at the assertion of the Lemma.
\end{proof}

\section{Continuity of $\gamma_1(\nu)$}
Rather simple symmetry-arguments have been used to prove that $\gamma_1(\nu) \leq \lambda_1$. For the exact evaluation of $\gamma_1(\nu)$ a deeper analysis is necessary. One natural approach consists of reducing the problem to 'simple' jump distributions by certain continuity considerations. The fact that the equilibrium measure $\mu^{\nu}$ depends continuously on $\nu$ with respect to the weak topology raises the question, whether the same is true for the spectral gap $\lambda(\nu)$. The answer is affirmative and we want to stress that the results of this chapter are valid in the multi-dimensional situation.

Instead of dealing with $\mathcal{P}_t^{\nu}$ directly, we investigate the behavior of the semigroup $\mathcal{S}^{\nu}_t$ defined in (\ref{dualsemigroup}), where we know that $\mathcal{S}^{\nu\ast}_t=\mathcal{P}_t^{\nu}$. The main reason for this procedure is that in contrast to $\mathcal{P}_t^{\nu}$, $\mathcal{S}^{\nu}_t$ turns out to be strongly continuous (see \cite{BaP1} and \cite{BaP2}) so that the general theory of strongly continuous semigroups (see e.g. \cite{EN}) can be applied. For example, this makes it possible to verify compactness for $S^{\nu}_t$
and to relate its spectrum to the spectrum of its generator by application of known results of semigroup theory (see e.g. \cite{BaP2}).   
The following two basic results from operator-theoretic perturbation theory, which are due to Kato, will be applied later on to the semigroup $\mathcal{S}^{\nu}_t$.
\begin{theorem}[\cite{Kato}, p.178]\label{Kato0}
Let $\mathcal{T}$ be a closed operator in the Banach space $X$ and let $\sigma (\mathcal{T})$ be separated into two parts $\Sigma'(\mathcal{T})$ and $\Sigma''(\mathcal{T})$ by a closed curve $\Gamma$, i.e. a rectifiable, simple closed curve $\Gamma$ can be drawn so as to enclose an open set containing $\Sigma'(\mathcal{T})$ in its interior and $\Sigma''(\mathcal{T})$ in its exterior. Then one has a decomposition of $\mathcal{T}$ according to the decomposition of $X=M'(\mathcal{T})\oplus M''(\mathcal{T})$ of the space in such a way that the spectra of the parts $\mathcal{T}_{M'}$ and $\mathcal{T}_{M''}$ coincide with $\Sigma'(\mathcal{T})$ and $\Sigma''(\mathcal{T})$, respectively and $\mathcal{T}_{M'} \in B(X)$.
\end{theorem}
In the following $\hat{\delta}(\cdot,\cdot)$ denotes the metric which defines the generalized convergence of closed operators (compare \cite{Kato} p. 202). 
\begin{theorem}[\cite{Kato}, p.212]\label{Kato}
Let $\mathcal{T}$ be a closed operator in the Banach space $X$ and let $\sigma (\mathcal{T})$ be separated into two parts $\Sigma'(T)$ and $\Sigma''(\mathcal{T})$ by a closed curve $\Gamma$, i.e. a rectifiable, simple closed curve $\Gamma$ can be drawn so as to enclose an open set containing $\Sigma'(\mathcal{T})$ in its interior and $\Sigma''(\mathcal{T})$ in its exterior. Let $X=M'(\mathcal{T}) \oplus M''(\mathcal{T})$ be the associated composition of $X$. Then there exists a $\delta > 0$, depending on $\mathcal{T}$ and $\Gamma$ with the following properties. Any closed operator $\tilde{\mathcal{T}}$ on $X$ with $\hat{\delta}(\tilde{\mathcal{T}},\mathcal{T})< \delta$ has spectrum $\sigma(\tilde{\mathcal{T}})$ likewise separated by $\Gamma$ into two parts $\Sigma'(\tilde{\mathcal{T}})$ and $\Sigma''(\tilde{\mathcal{T}})$. In the associated decomposition $X=M'(\tilde{\mathcal{T}}) \oplus M''(\tilde{\mathcal{T}})$, $M'(\tilde{\mathcal{T}})$ and $M''(\tilde{\mathcal{T}})$ are respectively isomorphic with $M'(\mathcal{T})$ and $M''(\mathcal{T})$. In particular $\dim M'(\tilde{\mathcal{T}})=\dim M'(\mathcal{T})$ and $\dim M''(\tilde{\mathcal{T}})  = \dim M''(\mathcal{T})$ and both $\Sigma(\tilde{\mathcal{T}})$ and $\Sigma''(\tilde{\mathcal{T}})$ are non-empty if this is true for $\mathcal{T}$. The decomposition $X=M'(\tilde{\mathcal{T}}) \oplus M''(\tilde{\mathcal{T}})$ is continuous in $\tilde{\mathcal{T}}$ in the sense that the projection $P[\tilde{\mathcal{T}}]$ of $X$ onto $M'(\tilde{\mathcal{T}})$ along $M''(\tilde{\mathcal{T}})$ tends to $P[\mathcal{T}]$ in norm as $\hat{\delta}(\tilde{\mathcal{T}},T) \rightarrow 0$. 
\end{theorem}
In order to apply Theorem \ref{Kato} observe that for a bounded operator $T$ and a sequence of closed operators $(\mathcal{T}_n)_n$ we have by Theorem 2.23 in \cite{Kato} that $\lim_{n \rightarrow \infty}\hat{\delta}(\mathcal{T}_n , \mathcal{T})=0$ if and only if the operators $\mathcal{T}_n$ are bounded for $n$ large enough and $\lim_{n \rightarrow \infty}\|\mathcal{T}_n -\mathcal{T}\| = 0$. 
Now, Theorems \ref{Kato0} and \ref{Kato} immediately imply the following conclusion. 
\begin{Corollary}\label{CKato}
Assume that $(\mathcal{T}_n)_{n \in \mathbb{N}}$ is a sequence of compact operators in a Banach space $X$ which converges with respect to the operator norm to a (necessarily compact) operator $\mathcal{T}$ and let $\Gamma$ be a curve such that $\Gamma$ divides the spectrum $\Sigma(\mathcal{T})$ of $\mathcal{T}$ into a two parts, one finite part $\Sigma'(\mathcal{T})$ lying inside the curve with no limit point and one part $\Sigma''(\mathcal{T})$ lying outside this curve. Then for $n$ large enough the spectra $\Sigma(\mathcal{T}_n)$ of the operators $\mathcal{T}_n$ are also separated by $\Gamma$ and since $\dim M'(\mathcal{T}_n) = \dim M'(\mathcal{T})< \infty$ and since $\mathcal{T}_nP[\mathcal{T}_n]$ converges in operator norm to $\mathcal{T}P[\mathcal{T}]$ as $n \rightarrow \infty$ we also have $\Sigma'(T_n) \rightarrow \Sigma'(\mathcal{T})$ as $n \rightarrow \infty$ with respect to the Hausdorff metric and in the following sense: 
\begin{displaymath}
\forall \varepsilon > 0\,\exists N \in\mathbb{N} \,\forall\lambda \in \Sigma'(\mathcal{T})\,\forall n \geq N \,\exists \lambda_1^n,\dots,\lambda_{k_n}^n \in \Sigma'(\mathcal{T}_n): \max_{l\in 1,...,k_n}|\lambda-\lambda^n_{l}|< \varepsilon.
\end{displaymath}
Furthermore, If $m$ ($m_{l}^n$) denotes the algebraic multiplicity of the eigenvalue $\lambda$ ($\lambda^n_{l}$) of $\mathcal{T}$ ($\mathcal{T}_n$), then $\sum_{l=1}^{k_n}m^n_{l}=m$.
\end{Corollary}
Now we turn to the application of these abstract results to the situation at hand. We want to stress that several of the following arguments are taken from \cite{BaP2}. We give rather detailed arguments since some care concerning the dependence of various estimates on the jump distribution is necessary.
 
We will need the following Lemmata. These are probably well-known, but we have not been able to locate them in the literature in the form we use them.
\begin{lemma}\label{uniformconvergence}
Let $D \subset \mathbb{R}^d$ be a bounded domain and let as above $(\lambda_n^D)_{n \in \mathbb{N}_0}$ and $(\varphi_n^{D})_{n \in \mathbb{N}_0}$ denote the sequence of eigenvalues (counted according to their multiplicities) and the associated eigenfunctions of the symmetric operator $L$
\begin{displaymath}
L:=\frac{1}{2}\sum_{i,j=1}^{d}\partial_i (a_{ij}(\cdot)\partial_j)
\end{displaymath}
in $D$ with Dirichlet boundary conditions, respectively. Then for any $\varepsilon > 0$ we have
\begin{equation}\label{sumfinite}
 \sum_{k=0}^{\infty}e^{-\lambda^D_k\varepsilon}\|\varphi^D_k\|_{\infty}^2 < \infty.
\end{equation}
\end{lemma}
\begin{proof}
By Jensen's inequality we have 
\begin{displaymath}
 (\varphi_k^D(x))^2= e^{2\lambda_kt}\bigl(\mathbb{E}_x\bigl[\varphi^{D}_k(B_t),\tau_D>t\bigr]\bigr)^2 \leq e^{2\lambda_kt}\mathbb{E}_x\bigl[\varphi_k^{D}(B_t)^2,\tau_D>t\bigr].
\end{displaymath}
Since $\sup_{x,y \in D}p^{D}(t,x,y) \leq t^{-d/2}$ (see e.g. \cite{Ar}), we conclude that 
\begin{equation}\label{subexponentional}
 \|\varphi_k^{D}\|_{\infty}^2 \leq e^{2\lambda_kt} \,t^{-d/2}\,\|\varphi_k^{D}\|_{L^2(D)}^2 = e^{2\lambda_kt} \,t^{-d/2}.
\end{equation}
Choosing $t=\frac{\epsilon}{4}$ in \eqref{subexponentional}, the finiteness of the sum in (\ref{sumfinite}) follows from Weyl's theorem, which states that for appropriate constants $c_1,c_2>0$ one has $c_1 k^{2/d} \leq \lambda^D_k \leq c_2 k^{2/d}$ (see e.g. \cite{D}, Theorem 6.3.1).
\end{proof}
A simple application of Lemma \ref{uniformconvergence} and the dominated convergence theorem gives
\begin{Corollary}\label{contofdens}
Let $D \subset \mathbb{R}^d$ be a bounded domain with smooth boundary and $L$ a symmetric diffusion operator with Dirichlet boundary conditions. Let $T_D = \inf\lbrace t>0 \mid B_t \in \partial D\rbrace$ the first hitting time of the boundary of $D$ the $L$-diffusion process $(B_t)_{t \geq 0}$ starting in $x \in D$ then the distribution of $T_D$ is absolutely continuous with a density which is jointly continuous in $(x,t) \in D\times (0, \infty)$.
\end{Corollary}
Let $L_0 = \frac{1}{2}\sum_{i,j=0}^{d}\partial_i\bigl(a_{ij}\partial_j\bigr)$ be the generator of a uniformly elliptic reversible diffusion in $\mathbb{R}^d$ with smooth and bounded coefficients written in divergence form or 
\begin{displaymath}
L_0 = \frac{1}{2}\sum_{i,j=1}^{d}a_{ij}\partial_{ij} + \bigl(\frac{1}{2}\sum_{i=0}^d\partial_i a_{ij}\bigr)_{j=0}^d\cdot \nabla =: \frac{1}{2}\sum_{i,j=1}^{d}a_{ij}\partial_{ij} + \sum_{i=0}^d\beta_i\cdot \nabla 
\end{displaymath}
in non-divergence form. We have seen that the first hitting time distribution of the set $\partial D$ has a density which is continuous with respect to the starting point of the $L_0$-diffusion. We will now use this result in order to extend this to diffusions of the form 
\begin{displaymath}
L := L_0 + \gamma \cdot \nabla,
\end{displaymath}
with $\gamma \in (C_c(\mathbb{R}^d))^{d}$. 
\begin{lemma}\label{contofdens2}
Let $\mathbb{P}_x$ denote the law of a $L$-diffusion and let as above $T_D=\inf\lbrace t > 0 \mid B_{t} \in \partial D\rbrace$ denote the first hitting time of $\partial D$. Then $\mathbb{P}_x(T_D \in \cdot\bigr)$ is absolutely continuous with a density $h(\cdot,\cdot)$ that is jointly continuous in
$(x,t)\in D \times (0,\infty) $. 
\end{lemma}
\begin{proof}
Let us denote by $\mathbb{P}_x$ the path space measures associated to the diffusion generated by $L$ and let $\mathbb{P}^{0}_x$ denote the path space measures corresponding to the $L_0$-diffusion. Then by the Girsanov theorem (see e.g. section 1.9 in \cite{P95}) one has
\begin{displaymath}
\frac{d\mathbb{P}_x}{d\mathbb{P}_x^{0}}\restriction_{\mathcal{F}_t} = N_t,
\end{displaymath}
where 
\begin{displaymath}
N_t := \exp\biggl( \int_0^ta^{-1} \gamma (X_s)\,d\bar{X}_s - \frac{1}{2}\int_0^t\langle a^{-1}\gamma (X_s),\gamma(X_s) \rangle_{\mathbb{R}^d}\,ds\biggr),
\end{displaymath}
where $\bar{X}_t = X_t -\int_0^{t}\beta(X_s)\,ds$.
Therefore we have 
\begin{displaymath}
\mathbb{P}_x\bigl(T_{D} > t\bigr)= \mathbb{E}^{0}_x\bigl[N_t;T_{D}>t\bigr],
\end{displaymath}
since we already know that the exit time distribution $\mathbb{P}^0_x(T_D \in \,\cdot)$ of the reversible diffusion corresponding to $L_0$ is absolute continuous with respect to the Lebesgue measure the same holds true for $\mathbb{P}_x\bigl(T_D \in \,\cdot\bigr)$. Thus there is a function $h(x,t)$ ($x \in D, t>0$), which is jointly measurable such that for $x \in D$ and measurable $A \subset (0,\infty)$  
\begin{displaymath}
\mathbb{P}_x(T_D \in A)=\int_A h(x,s)\,ds.
\end{displaymath}
It remains to show that $h(x,t)$ has a continuous modification in $(x,t)\in D \times (0,\infty)$. As a first step let us show that $h(x,t)$ is a.s. uniformly bounded in $x$ for all $t\ge c>0$,
i.e. $h(x,t)\le M_{t}\le M_{c}<\infty\,\,\forall t\in[c,\infty]$. We will use Lebesgue's version of the fundamental theorem of calculus to conclude that for almost every $(x,t) \in D \times (\varepsilon,\infty)$ 
\begin{equation}\label{e:hmitlebesgue}
\lim_{r,u \rightarrow  0}\frac{1}{2u}\frac{1}{|B(x,r)|}\int_{B(x,r)}\mathbb{P}_x\bigl( T_D  \in (t-u, t+u)\bigr)\,dx = h(x,t)
\end{equation}
as well as
\begin{equation}\label{e:mappingprop}
\begin{split}
\frac{1}{2u}\frac{1}{|B(x,r)|}&\int_{B(x,r)}\mathbb{P}_x\bigl( T_D  \in (t-u, t+u)\bigr)\,dx \\
&= \frac{1}{2u}\frac{1}{|B(x,r)|}\int_{B(x,r)}\int_D p^D(t-\varepsilon,x,y)\mathbb{P}_y\bigl( T_D \in (\varepsilon-u,\varepsilon+u)\bigr)\,dy\,dx \\
&= \frac{1}{2u}\frac{1}{|B(x,r)|}\int_{B(x,r)}\int_D p^D(t-\varepsilon,x,y)\mathbb{P}_y\bigl( T_D \in (\varepsilon-u,\varepsilon+u)\bigr)\,dy\,dx\\
&= \frac{1}{2u}\frac{1}{|B(x,r)|}\int_{\varepsilon-u}^{\varepsilon+u}\int_{B(x,r)}\int_D p^D(t-\varepsilon,x,y)h(y,s)\,dy\,dx\,ds\\
&\,\rightarrow \int_D p^D(t-\varepsilon,x,y)h(y,\varepsilon)\,dy
\end{split}
\end{equation}
as $r,u\rightarrow 0+$. Using equations \eqref{e:hmitlebesgue} and \eqref{e:mappingprop} together with the fact that for very positive $t >0$ the operaror $P^D_t$ maps an integrable function to a bounded one we deduce that for some constant $C>0$ 
\begin{eqnarray}\label{integrand}
|h(x,t)-h(z,s)|&=&\biggl|\int_{D}(p^D(t-u,x,y)-p^D(s-u,z,y))h(y,u)dy \biggr|\nonumber\\
&\le& C  \int_{D}|p^D(t-u,x,y)-p^D(s-u,z,y)|\,dy.
\end{eqnarray}
Due to continuity of $p^D(\cdot,\cdot,y)$, we see that for $(z,s)\rightarrow (x,t) \in D \times (0,\infty)$ the integrand in \eqref{integrand} converges pointwise to zero,
and hence we obtain by Lebesgue's theorem that
\[
\lim_{(z,s)\rightarrow(x,t)}|h(x,t)-h(z,s)|=0,
\] 
which yields the claim.

\end{proof}
\begin{lemma}\label{l:unifconvheatk}
Let $(\nu_n)_{n \in \mathbb{N}}$ a sequence of probability measures in $D$ converging weakly to the probability measure $\nu$. Then for every $\varepsilon > 0$  
\begin{displaymath}
\lim_{n\rightarrow \infty}\sup_{s>\varepsilon, y \in D}\bigl|p^D(s,\nu_n,y)-p^D(s,\nu,y)\bigr| = 0
\end{displaymath}
\end{lemma}
\begin{proof}
First observe that due to $\lim_{n\rightarrow \infty}\nu_n = \nu$ weakly we have for every $s>0$ and $z \in D$ 
\begin{displaymath}
\lim_{n\rightarrow \infty}p^D(s,\nu_n,z) = p^D(s,\nu,z)
\end{displaymath}
as well as 
\begin{displaymath}
\lim_{n \rightarrow \infty} \int_Dp^D(s,\nu_n,z)\,dz = \int_Dp^D(s,\nu,z)\,dz.
\end{displaymath}
Thus as a consequence of Scheff\'e's lemma we conclude that
\begin{equation}\label{e:scheffea}
\lim_{n \rightarrow \infty}\int_D\bigl| p^D(s,\nu_n,z)-p^D(s,\nu,z)\bigr|\,dz=0. 
\end{equation}
Observe now that due to the Chapman-Kolmogorov equations
\begin{equation*}
\begin{split}
\big| p^D(\varepsilon + s,\nu_n,z)&-p^D(\varepsilon+s,\nu,z)\bigr|= \biggl|\int_D \bigl(p^D(\varepsilon/2 ,\nu_n,y)-p^D(\varepsilon/2,\nu,y)\bigr)p^D(\varepsilon/2+s,y,z)\,dy\biggr| \\
&\leq \bigl[\sup_{s\geq 0;y,z\in D}p^D(\varepsilon + s,y,z)\bigr]\int_D\bigl| p^D(\varepsilon/2 ,\nu_n,y)-p^D(\varepsilon/2,\nu,y)\bigr|dy.
\end{split}
\end{equation*}
Since according to equation \eqref{e:scheffea} the right hand side converges to $0$ the assertion of the Lemma is shown. 
\end{proof}
\begin{lemma}\label{BaP2lemma}
Assume that $\nu_{i}\stackrel{w}{\rightarrow} \nu$. Then there exists for each $t\ge0$ an $\alpha=\alpha(t) < 1$ such that for all $\tilde{\nu} \in \lbrace  \nu,\nu_1,\nu_2,\dots\rbrace$ and $x \in D$ 
\begin{displaymath}
 \mathbb{P}_x\bigl(T^{\nu}_{i+1}\leq t) \leq \alpha^i\,\,\,\forall i\in\mathbb{N}.
\end{displaymath}
\end{lemma}
\begin{proof}
For  $\tilde{\nu} \in \lbrace  \nu,\nu_1,\nu_2,\dots\rbrace$ define
\[c_{\tilde{\nu}}=c_{\tilde{\nu}}(t)=1-\int_D\int_D p^D(t,x,y)\,dy\, \tilde{\nu}(dx).\] 
As a direct consequence of the Markov property we get for $\tilde{\nu} \in \lbrace \nu,\nu_1,\nu_2,\dots\rbrace$
\begin{equation}\label{finalconclusion}
\begin{split}
 \mathbb{P}_x\bigl(T^{\tilde{\nu}}_{i+1}\leq t) &\leq \bigl(\mathbb{P}(\exists \, 0 \leq s\leq t: W^{\tilde{\nu},1}_s \in \partial D)\bigr)^i \\
 &=c_{\tilde{\nu}}^{i}.
\end{split}
\end{equation}
>From $\lim_{n\rightarrow \infty}\nu_n = \nu$ it follows that for all $t>0$ we have
\begin{displaymath}
c_{\nu_n}\rightarrow c_{\nu} \mbox{ for }n\rightarrow\infty,
\end{displaymath}
which implies that for all $\epsilon>0$ there exists $n(\epsilon)$ with the property that for $n \geq n(\epsilon)$ 
\begin{equation}\label{less1}
c_{\nu_n}\le c_{\nu}+\epsilon.
\end{equation}
For $\epsilon$ small enough (e.g. $\epsilon\le \frac{1-c_{\nu}}{2}$) the right hand side in (\ref{less1}) is strictly smaller than one.
Since 
\begin{displaymath}
c_{\nu_n} < 1\mbox{ for all }n\in\{1,2,\ldots,n(\epsilon)\}.
\end{displaymath}
Let
\[
\alpha=\max\{c_{\nu_1},c_{\nu_2},\ldots,c_{\nu_{n(\epsilon)}},c_{\nu}+\epsilon\}.
\]
Obviously, $\alpha=\alpha(t)<1$ and since $c_{\tilde{\nu}}\le\alpha$, the result follows from (\ref{finalconclusion}).
\end{proof}
Before we state and prove the main result of this section, let us introduce important notations, which are necessary for illustrating the structure of the proof.  
Let
\begin{equation}
h^{\rho}(t)=\frac{d}{dt}\mathbb{P}_{\rho}\bigl(T_{1}^{\nu}<t\bigr)
\end{equation}
and for $n \ge 2$
\begin{equation}
h^{\rho,\nu}_n(t)=\frac{d}{dt}\mathbb{P}\bigl(T_{n}^{\rho,\nu}<t\bigr)=\bigl( h^{\rho}\ast(h^{\nu})^{\ast,n-1}\bigr)(t).
\end{equation}
In the case when $\rho=\delta_{x}$ is the Dirac measure in $x$, we will write $h^{x,\nu}_n(t)$ instead of $h^{\delta_x,\nu}_n(t)$. Observe that $h^{\rho}(t)$ is nothing else than the density of the first exit time of the $L$-diffusion with the initial distribution $\rho$ from the domain $D$.
Moreover, set for $n \geq 1$
\begin{displaymath}
 \mathcal{S}^{\nu}_{0,t}g(z):= \int_D p^{D}(t,x,z)g(x)\,dx\,\text{   and   }\, \mathcal{S}^{\nu}_{n,t}g(z):=\int_D \int_0^{t}p^D(t-s,y,z)g(y)h^{x,\nu}_n(s)\,ds\,\nu(dy)dx
\end{displaymath}
and similarly
\begin{displaymath}
 \mathcal{P}^{\nu}_{0,t}f(x):= \int_D p^{D}(t,x,y)f(y)\,dy\,\text{   and   }\, \mathcal{P}^{\nu}_{n,t}f(x):=\int_D \int_0^{t}p^D(t-s,y,z)f(z)h^{x,\nu}_n(s)\,ds\,\nu(dy)dz.
\end{displaymath}
With these notations let us recall the definition of $P_t^{\nu}$ and write it in another way. For all $f\in L^{\infty}$ we have
\begin{equation}\label{GK1}
\begin{split}
\mathcal{P}_t^{\nu}f(x)&=  \mathbb{E}_x\bigl[f(X_t^{\nu})\bigr] = \mathbb{E}_x\bigl[f(X_t^{\nu}),T^{\nu}_1>t\bigr] + \sum_{n=1}^{\infty}\mathbb{E}_x\bigl[f(X_t^{\nu}),\,T^{\nu}_n \leq t < T^{\nu}_{n+1}\bigr] \\
  &= \mathbb{E}_x\bigl[f(X^{\nu}_t), T_1^{\nu}> t\bigr] + \sum_{n=1}^{\infty}\int_{0}^{t}\mathbb{E}_{x} \bigl[f(X_t^{\nu}),T^{\nu}_{n+1}>t|T^{\nu}_{n}=s\bigr]\mathbb{P}_x\bigl(T^{\nu}_{n}\in ds\bigr) \\
  &=\mathbb{E}_x\bigl[f(X^{\nu}_t),T^1_{\nu} > t\bigr] + \sum_{n=1}^{\infty}\int_{0}^{t}\mathbb{E}_{\nu}\bigl[f(X_{t-s}^{\nu});S^{\nu}_{n+1}>t-s\bigr]h_n^{x,\nu}(s)ds\\
 &=\int_D p^{D}(t,x,y)f(y)\,dy  + \sum_{n=1}^{\infty}\int_{0}^{t}\int_D \int_D p^D(t-s,y,z)f(z)\nu(dy)dz \,h^{x,\nu}_n(s)\,ds\,  \\
 &=\mathcal{P}^{\nu}_{0,t}f(x)+\sum_{i=1}^{\infty}\mathcal{P}^{\nu}_{i,t}f(x).
 \end{split}
\end{equation}
and hence
\begin{equation}\label{PartititionP}
\mathcal{P}_t^{\nu}=\sum_{i=0}^{\infty}\mathcal{P}^{\nu}_{i,t}.
\end{equation}
A dual version of this calculation immediately yields
\begin{equation}\label{PartititionS}
\mathcal{S}^{\nu}_{t}=\sum_{i=0}^{\infty}\mathcal{S}^{\nu}_{i,t}.
\end{equation}
Moreover, it is obvious that
\begin{equation}
\langle \mathcal{S}^{\nu}_{i,t}g,f\rangle=\langle g,\mathcal{P}_{i,t}^{\nu}f\rangle\,\,\forall g\in L^{1},\,f\in L^{\infty}, \,\,i\in \mathbb{N}\cup\{0\},
\end{equation}
and hence 
\[\mathcal{S}^{\nu\ast}_{i,t}=\mathcal{P}_{i,t}^{\nu}\,\,\forall i\in\mathbb{N}\cup\{0\}.\] 
\\
\\
We are now able to prove our first main result. During the proof we use $\|\,\cdot\,\|_{p,p}$ to denote the operator norm of a bounded linear operator mapping $L^p$ to $L^p$.
\begin{theorem}\label{contofgap}
The spectral gap $\gamma_1(\nu)$ of the generator of the semigroup $(\mathcal{S}^{\nu}_t)_{t \geq 0}$ acting in $L^1(D)$ depends continuously on the jump distribution $\nu$ with respect to the weak topology. 
\end{theorem}

Several ideas of the proof are taken from \cite{BaP2}, where they are used for a different purpose.

\begin{proof}
We will use Kato's results cited in the beginning of this section. In order to apply these results we have to show that for admissable distributions $\nu,\nu_1,\nu_2,\dots$ 
\begin{displaymath}
 \lim_{n\rightarrow \infty}\nu_n = \nu\, \text{ with respect to the weak topology }\,\Rightarrow\,\lim_{n\rightarrow \infty}\|\mathcal{S}^{\nu_n}_t - \mathcal{S}^{\nu}_t\|_{1,1}=0.\,\,\,\mbox{ for all }t\in\mathbb{R}_{+}
\end{displaymath}

We have 
\begin{displaymath}
 \|\mathcal{S}^{\nu}_t-\mathcal{S}^{\nu_n}_t\|_{1,1}\leq \underbrace{\bigl\|\mathcal{S}^{\nu}_t - \sum_{i=0}^N \mathcal{S}^{\nu}_{i,t} \bigr\|_{1,1}}_{I(N,\nu)}+\underbrace{\bigl\|\sum_{i=0}^N(\mathcal{S}^{\nu}_{i,t}-\mathcal{S}^{\nu_n}_{i,t})\bigr\|_{1,1}}_{II(N,\nu,\nu_n)} + \underbrace{\bigl\|\mathcal{S}^{\nu_n}_t - \sum_{i=0}^N \mathcal{S}^{\nu_n}_{i,t} \bigr\|_{1,1}}_{III(N,\nu_n)}. 
\end{displaymath}
By duality and Lemma \ref{BaP2lemma} we obtain
\begin{eqnarray}\label{almosttrue}
I(N,\nu)&=&\bigl\|\sum_{i=N+1}^{\infty}\mathcal{S}^{\nu}_{i,t}\bigr\|_{1,1}= \bigl\|\sum_{i=N+1}^{\infty}\mathcal{P}^{\nu}_{i,t}\bigr\|_{\infty,\infty}\nonumber\\
&=&\sup_{f\in L^{\infty}:||f||_{\infty}=1}\mathbb{E}_{x}\bigl(f(X_t^{\nu});T_{N+1}^{\nu}\le t\bigr)\nonumber\\
&\le&\sup_{x}\mathbb{P}_{x}\bigl(T_{N+1}^{\nu}\le t\bigr)\le \alpha(t)^{N}.
\end{eqnarray}
Observe that by Lemma \ref{BaP2lemma} inequality (\ref{almosttrue}) holds for all $\tilde{\nu}\in\{\nu,\nu_1,\nu_2,\ldots,\nu_n,\ldots\}$ with $\nu_{n}\rightarrow\nu$ (in the weak topology). Hence, for all $\epsilon>0$ we can find $n_0=n_{0}(\epsilon)$ such that for $N\ge n_0$ we have
\begin{equation}
I(N,\nu)\le\epsilon,\,\,\,III(N,\nu_n)\le\epsilon.
\end{equation}    
Since $II(N,\nu,\nu_n)\le N\max_{i\in\{1,2,\ldots,n_0\}}\{\bigl\|\mathcal{S}^{\nu}_{i,t}-\mathcal{S}^{\nu_n}_{i,t}\bigr\|_{1,1}\}$,
it remains to show that for all $\epsilon>0$ there exists $n_1$ such that
\begin{displaymath}
\max_{i\in\{1,2,\ldots,n_0\}}\{\|\mathcal{S}^{\nu}_{i,t}-\mathcal{S}^{\nu_n}_{i,t}\|_{1,1}\}\le \frac{\epsilon}{N}\,\,\,\forall n\ge n_1.
\end{displaymath}
Since we maximize over finitely many elements, it is easy to see that it suffices to show that every $\epsilon>0$ and every $i_0\in\{1,2,\ldots\}$
there exists $n_1$ such that
\begin{equation}\label{none}
\|\mathcal{S}^{\nu}_{i_0,t}-\mathcal{S}^{\nu_n}_{i_0,t}\|_{1,1}\le \frac{\epsilon}{N}\,\,\,\forall n\ge n_1.
\end{equation}
We have by duality
\begin{eqnarray}
&&\|\mathcal{S}^{\nu}_{i_0,t}-\mathcal{S}^{\nu_n}_{i_0,t}\|_{1,1}=\|\mathcal{P}^{\nu}_{i_0,t}-\mathcal{P}^{\nu_n}_{i_0,t}\|_{\infty,\infty}\nonumber\\
&=&\sup_{x \in D}\sup_{f\in L^{\infty}:||f||_{\infty}=1}\biggl|\int_0^t\int_Dp^D(t-s,\nu,z)f(z)\,dyh^{x,\nu}_{i_0}(s)\,ds-\int_0^t\int_Dp^D(t-s,\nu_n,y)f(y)\,dz h_{i_0}^{x,\nu_n}(s)\,ds\biggr|\nonumber \\
&\le& \sup_{x \in D}\sup_{f\in L^{\infty}:||f||_{\infty}=1}\biggl|\int_0^t\int_Dp^D(t-s,\nu,z)f(y)\,dz h_{i_0}^{x,\nu}(s)\,ds-\int_0^t\int_Dp^D(t-s,\nu_n,z)f(z)\,dz h_{i_0}^{x,\nu}(s)\,ds\biggr| \nonumber \\
&&+\sup_{x \in D}\sup_{f\in L^{\infty}:||f||_{\infty}=1}\biggl|\int_0^t\int_D p^D(t-s,\nu_n,z)h_{i_0}^{x,\nu}\,ds f(z)\,dz- \int_D\int_0^t p^D(t-s,\nu_n,z)h_{i_0}^{x,\nu_n}(s)\,ds f(z)\,dz\biggr|\nonumber\\
&=&\sup_{x \in D}\sup_{f\in L^{\infty}:||f||_{\infty}=1} \biggl|\int_0^t\left(\int_D (p^D(t-s,\nu,z)-p^D(t-s,\nu_n,z))f(z)\,dz\right) h_{i_0}^{x,\nu}(s)\,ds\biggr|\nonumber \\
&&+
\sup_{x \in D}\sup_{f\in L^{\infty}:||f||_{\infty}=1}\biggl|\int_0^{t}\int_Dp^{D}(t-s,\nu_n,z)f(z)\,dz\bigl(h^{x,\nu}_{i_0}(s)-h^{x,\nu_n}_{i_{0}}(s)\bigr)\,ds\biggr| 
\nonumber
\\
&=&\tilde{I}_n(f) + \tilde{II}_n(f)\nonumber.
\end{eqnarray}
The first term can be estimated by
\begin{equation*}
 \begin{split}
  \tilde{I}_n(f) &\le \sup_{x \in D} \biggl|\int_0^t \underbrace{\left(\int_D |p^D(t-s,\nu,z)-p^D(t-s,\nu_n,z)|\,dz \right)}_{ a(n,t-s)} h_{i_0}^{x,\nu}(s)\,ds\biggr| \\
    &=  \sup_{x\in D}\biggl|\int_0^t a(n,t-s)h_{i_0}^{x,\nu}(s)\,ds\biggr|\\
    &\leq  \sup_{x\in D}\biggl|\int_0^{t-\varepsilon} a(n,t-s)h_{i_0}^{x,\nu}(s)\,ds\bigr| + \sup_{x\in D}\biggl|\int_{t-\varepsilon}^t a(n,t-s)h_{i_0}^{x,\nu}(s)\,ds\biggr|,
 \end{split}
\end{equation*}
where $\varepsilon \in (0,t)$ is arbitrary. Using $a(n,t-s) \leq 2$ for every $n \in \mathbb{N}$ and $\int_0^t h_{i_0}^{x,\nu}(s)\,ds \leq 1$ for every $x \in D$ we conclude that
\begin{equation}\label{firstterm}
\begin{split}
 \tilde{I}_n(f) &\leq \sup_{s \leq t-\varepsilon}a(n,t-s) + 2\sup_{x \in D}\int_{t-\varepsilon}^th_{i_0}^{x,\nu}(s)\,ds.
\end{split}
\end{equation}
Due to Lemma \ref{l:unifconvheatk} the first term in \eqref{firstterm} converges to $0$ for every $\varepsilon > 0$. Whereas the second term in \eqref{firstterm} can be made arbitrary small by chosing $\varepsilon$ small enough (see Lemma 3 in \cite{BaP2}). Thus
\begin{displaymath}
\limsup_{n \rightarrow \infty}\tilde{I}_n(f) = 0.
\end{displaymath} 

In order to estimate $\tilde{II}_n(f)$, let
\begin{equation}
 \tilde{h}^{x}_{j,i_0-j-1}=h^x * \underbrace{(h^{\nu} * \dots * h^{\nu})}_{j\mbox{ terms}} * \underbrace{(h^{\nu_n} * \dots *h^{\nu_n})}_{i_0-j-1\mbox{ terms}}.
\end{equation}
>From commutativity of the convolution it follows
\begin{eqnarray}
 \tilde{h}^{x}_{j,i_0-j-1}-\tilde{h}^x_{j-1,i_0-j}(s) &=& h^x * [\underbrace{h^{\nu} * \dots * h^{\nu}}_{j-1\mbox{ terms}} * (h^{\nu}-h^{\nu_n}) * \underbrace{h^{\nu_n} * \dots *h^{\nu_n}}_{i_0-j-1\mbox{ terms}}](s)\nonumber\\
 &=&(h^{\nu}-h^{\nu_n}) *h^x * [h^{\nu} * \dots * h^{\nu} * h^{\nu_n} * \dots *h^{\nu_n}](s)\nonumber\\
 &=&(h^{\nu}-h^{\nu_n})*\tilde{h}^{x}_{j-1,i_0-j-1}.
\end{eqnarray}

Therefore we have
\begin{eqnarray}\label{secondterm}
 \tilde{II}_n(f) &\le & \|f\|_{\infty}\sup_{x \in D} \int_0^t|h_{i_0}^{x,\nu}(s)-h_{i_0}^{x,\nu_n}(s)|\,ds\nonumber\\
 &\le&\sup_{x \in D}\sum_{j=0}^{i_0-2}\int_{0}^{t}|\tilde{h}^{x}_{j+1,i_0-j-2}-\tilde{h}^x_{j,i_0-j-1}|(s)\,ds \nonumber\\
 &=&\sup_{x \in D}\sum_{j=0}^{i_0-2}\int_{0}^{t}|(h^{\nu}-h^{\nu_n})*\tilde{h}^{x}_{j,i_0-j-2}|(s)\,ds \nonumber\\
 &\le& (i_0-1)\sup_{x \in D}\max_{j\in\{0,1,\ldots,i_0-2\}}\int_{0}^{t}|(h^{\nu}-h^{\nu_n})|*\tilde{h}^{x}_{j,i_0-j-2}(s)\,ds \nonumber\\
 &\stackrel{Fubini}{\le}&(i_0-1)\int_{0}^{t}|(h^{\nu}-h^{\nu_n})|(s)ds \times \sup_{x \in D}\int_{0}^{t}\tilde{h}^{x}_{j_0,i_0-j_0-2}(s)\,ds \nonumber\\
 &\le&(i_0-1)\int_{0}^{t}|(h^{\nu}-h^{\nu_n})|(s)ds\nonumber\\
&&\rightarrow 0\,\,\,\,\,\mbox{ for }n\rightarrow\infty,
\end{eqnarray}
where $j_0$ denotes the index where the maximum is attained. For the last step we use that according to Lemma \ref{contofdens2} the density 
\begin{displaymath}
h(x,s) = \frac{\mathbb{P}_{x}\bigl(T_D \in ds\bigr)}{ds} 
\end{displaymath}
of the first exit time $T_D$ from the domain $D$ is continuous in $x$, and hence by weak convergence the integrand of the inner integral converges pointwise to zero. Another application of Scheff\'e's lemma shows that the last integral in \eqref{secondterm} converges to zero. From (\ref{firstterm}) and (\ref{secondterm}) it follows that there exists $n_1$ such that 
(\ref{none}) holds.
\end{proof}

\begin{remark}
\text{   }
\begin{enumerate}
\item
In \cite{BaP1} Ben-Ari and Pinsky formulated the question, whether the spectral gap $\gamma_1(\nu)$ depends continuously on $\nu$. Theorem \ref{contofgap} answers this question affirmatively. But observe that Theorem \ref{contofgap} actually shows that arbitrary finite subsets counted according to multiplicities of the spectrum depend continuously on $\nu$. 
\item
In the recent preprint \cite{KM} we also give answers to Question 1 and Question 2, which are posed on page 130 \cite{BaP1}. 
\end{enumerate}
\end{remark}
The continuity of the spectral gap might be of some importance e.g. in optimization problems concerning the speed of convergence for the multi-dimensional BMJB. 
\begin{Corollary}
Let $D \subset \mathbb{R}^d$ satisfy our standard assumptions and let $\gamma_1(\nu)$ denote the spectral gap of the BMJB in $D$ with jump distribution $\nu$. Then for every precompact (with respect to the weak topology) subset $\mathcal{K}$ of probability measures on $D$ satisfying $\sup_{n \in \mathcal{K}}\gamma_1(\nu)<\infty$ there exists a jump-distribution $\nu_0$ such that 
\begin{displaymath}
\gamma_1(\nu_0) = \sup_{\nu \in \mathcal{K}}\gamma_1(\nu).
\end{displaymath}
\end{Corollary}
\begin{proof}
Take a sequence $(\nu_n)_n \subset \mathcal{K}$ such that $\lim_{n\rightarrow \infty}\gamma_1(\nu_n) = \sup_{\nu \mathcal{K}}\gamma_1(\nu) < \infty$. Then by precompactness of $\mathcal{K}$ we can extract a weakly convergent subsequence $(\nu_{n_k})_{k\in\mathbb{N}}$ with limit $\nu_0$ and due to the continuity of the spectral gap 
\begin{displaymath}
\gamma_1(\nu_0)= \lim_{k \rightarrow \infty}\gamma_1(\nu_{n_k}) = \sup_{\nu \in \mathcal{K}}\gamma_1(\nu).
\end{displaymath}
\end{proof}
For the remaining part of this work we note the following simple 
\begin{Corollary}\label{c:redcompsupp}
Let $D =(a,b)$ and let $\nu$ be an admissible jump distribution, i.e. $\nu(\lbrace a, b\rbrace)=0$ and set $\nu_n(\,\cdot\,) = \nu\bigl(\,\cdot\,\mid \lbrace x : \text{dist}(x,\partial D)> 1/n\rbrace\bigr)$. Then we have 
\begin{displaymath}
 \lim_{n\rightarrow \infty} \gamma_1(\nu_n) = \gamma_1(\nu).
\end{displaymath}
\end{Corollary}
Thus once it is shown, that $\gamma_1(\cdot)$ is constant on compactly supported jump distributions (or even on jump distributions supported on finite sets), we can conclude that $\gamma_1(\cdot)$ is constant on the set of all admissible jump distributions.

\section{Rate of Convergence: Probabilistic Approach}

In this section we consider BMJB on the interval $(a,b)$ and recover a recent result of \cite{LLR}, which was previously shown via elegant Fourier-analytic arguments. As already mentioned in \cite{LLR} these Fourier-analytic arguments do not offer any probabilistic explanation.
We use the coupling approach, which provides a probabilistically more satisfactory explanation for the obtained convergence rates. As is well-known the coupling method usually involves the construction of two suitably dependent processes as a crucial step. In our situation the main difficulty is present due to the fact that in contrast to one-dimensional diffusions without jumps two independent BMJB can 'pass' each other without hitting each other. The lower bound on the convergence rate is derived -- as usual in coupling approaches -- via an investigation of the tail behavior of the coupling time.
The following result constitutes the second main theorem of this work.  
\begin{theorem}\label{main}
Assume that $d=1$. Then we have
\begin{enumerate}
\item
$\gamma_{1}(\nu)=\lambda_1^D$
\item
There exists an efficient coupling.
\end{enumerate}
\end{theorem}
The proof is split into smaller pieces.
\subsection{Auxiliary results} 
We will need the following elementary auxiliary results:
\begin{proposition}\label{firstprop}
Let $I=(a,b)$ be an open interval with center $c=\frac{a+b}{2}$. For $y\in I$ let $\tau_{y}=\inf\{t:y+B_t\in\partial I\}$ the first time
of leaving the interval. Then we have for all $y\in I$ and $t\in\mathbb{R}_+$
\[
\mathbb{P}\bigl(\tau_{y}>t\bigr)\le \mathbb{P}\bigl(\tau_{c}>t\bigr).
\] 
\end{proposition}
\begin{proof}
The proof is based on a simple coupling argument: Without loss of generality we may assume that $b>y>c$. We define the coupling as follows:
For $t<\tau^{cy}=\inf\{t:X_t=Y_t\}$ let
$X_t=y-B_t$ and $Y_t=c+B_t$. Now let us distinguish the following two cases:
First case: $X$ and $Y$ meet in $\frac{c+y}{2}$ at time $\tau^{cy}<\tau_c$. In this case let
$Y=X=\frac{c+y}{2}+B_t -B_{\tau_c^1}$ for $t\ge \tau^{cy}$ and hence both processes leave the interval at the same time.\\
Second case: $X$ and $Y$ do not meet each other before $\tau_c$. In this case, we have by definition of the processes that
$\tau_y <\tau_c$. Hence we obtain
\begin{eqnarray*}
\mathbb{P}\bigl(\tau_c>t\bigr)&=&\mathbb{P}\bigl(\tau_c>t, \tau_c>\tau^{cy}\bigr)+\mathbb{P}\bigl(\tau_c>t,\tau_c<\tau^{cy}\bigr)\nonumber\\
&=&\mathbb{P}\bigl(\tau_y>t, \tau_c>\tau^{cy}\bigr)+\mathbb{P}\bigl(\tau_c>t,\tau_c<\tau^{cy}\bigr)\nonumber\\
&\ge&\mathbb{P}\bigl(\tau_y>t, \tau_c>\tau^{cy}\bigr)+\mathbb{P}\bigl(\tau_y>t,\tau_c<\tau^{cy}\bigr)\nonumber\\
&=&\mathbb{P}\bigl(\tau_y>t\bigr).
\end{eqnarray*}
 \end{proof}
Now we come back to the symmetry argument, which already played an essential idea in Proposition \ref{bound}. In the following Proposition we put this symmetry argument -- which is the key idea in our approach -- in the coupling context. 
\begin{proposition}\label{keyprop}
Suppose that $a<x\le c=\frac{a+b}{2}\le y<b$ and $y=R(x)=a+b-x$. Then there is a coupling of $X^{\nu,x}$ and $X^{\nu,y}$ such that the coupling
time is equal to the exit time of a standard BM (=starting at $0$) from the interval $\tilde{I}=\tilde{I}(x,y)=(-\frac{y-x}{2},\frac{(b-a)-(x-y)}{2})$ .
\end{proposition}
\begin{proof}
We use the same notation as in the previous proof and set $\tau=\tau_x\wedge\tau^{xy}$. For $t<\tau$ define
 $X^{\nu,x}=x+B_t$ and $X^{\nu,y}=y-B_t$. Since $y=R(x)$, we have to distinguish \\
\begin{itemize}
\item
First Case:\\
$X^{\nu,x}_{\tau}=x+B_{\tau}=y-B_{\tau}=X^{\nu,y}_{\tau}=c$. For $t\ge\tau$ we define $X^{\nu,x}_{t}=X^{\nu,y}_{t}=c+B_{t-\tau}$. 
\item
Second Case:\\
$x+B_{\tau}=a$ and $y-B_{\tau}=b$. For $t\ge\tau$ we define $X^{\nu,x}_{t}=X^{\nu,y}_{t}=\nu+B_{t-\tau}$
\end{itemize}
We see that as long as the path of $B_t$ is contained in the interval $\tilde{I}=(-\frac{y-x}{2},\frac{(b-a)-(x-y)}{2})$, the
processes $X^{\nu,x}$ and $X^{\nu,y}$ do not merge, but once the Brownian motion $(B_t)_{t \geq 0}$ exits $\partial\tilde{I}$ they immediately colesce. This yields the claim.   
\end{proof}
\begin{remark}
A crucial observation is that $|J(x,R(x))|=\frac{b-a}{2}$ is independent of $x$ and that $\lambda_{1}^{(a,b)}=\lambda_{0}^{(a,\frac{a+b}{2})}$.  
\end{remark}
Now we will show that this idea can be extended to processes with arbitrary initial values $x,y$ whenever the jump measure is compactly supported 
in $(a,b)$.
\subsection{Construction of the coupling}

Now let us built up the coupling. First of all note that by a very simple argument using the triangle inequality and a symmetry argument
we can assume without loss of generality that
\begin{equation}\label{key_assump}
0<y-x<dist(supp(\nu),\{a,b\})\, \text{   and   }\,\frac{a+b}{2} \leq \frac{x+y}{2}.
\end{equation} 
Let us introduce two copies of BMJB, called $X$ and $Y$ in the following way:
\begin{enumerate}
\item[a)]
Let $x_1=x$, $y_1=y$ and $X_t=x_1+B_t$, $Y_t=y_1+B_t$. Now stop stage a) at time $\tau_1=\tau_{sym}\wedge \tau_{b}$,
where $\tau_{sym}=\inf\{t:R(x_1+B_t)=y_1+B_t\}$ and $\tau_{b}=\inf\{t:y_1+B_t=b\}$, 
i.e. we stop when either the copies are in a symmetric position ($\tau_1=\tau_{sym}$) or when $Y$ hits $b$ ($\tau_1=\tau_{b}$). Note that due to the assumption \eqref{key_assump} one of these two cases has to occur.\\
\textbf{First case:} If $\tau_1=\tau_{sym}$ then at time $\tau_1$ we are in the situation to apply the coupling presented in Proposition \ref{keyprop}.
\\
\textbf{Second case:} If $\tau_1=\tau_{b}$ then we have $x_2:=X_{\tau_1}=X_{\tau_1-}=x_1 +(b-y)=b-(y-x)$ and $Y_{\tau_1}=y_2=J_1$, where $J_1$ denotes the first jump with distribution $\nu$.  The construction is continued with the next stage.
\item[b)]
Let $X=x_2+(B_t-B_{\tau_1})(=x_1+B_t)$ and 
\begin{equation}\label{stoch_integral}
Y=y_2-(B_t-B_{\tau_1})=\int_{0}^{t}(-1)^{\mathbf{1}_{\lbrace\tau_1<t \rbrace}}dB_s. 
\end{equation}
To see that the stochastic integral in the definition of $Y$ is well-defined, let  
$\mathcal{F}_t=\sigma(B_s,s\le t)$ and $\mathcal{G}_t=\sigma(\mathcal{F}_t,\sigma(J_1))$ the $\sigma$-field generated by $\sigma(J_1)$ and $\mathcal{F}_t$. Now observe that $(-1)^{\mathbf{1}_{\{\tau_1<t\}}}$ is $\mathcal{G}_t$-measurable and that by independence of $\sigma(J_1)$ and $\mathcal{F}_t$ $B_t$ is still a Brownian motion with respect to $\mathcal{G}_t$, the integral in \eqref{stoch_integral} is still well-defined. 
Now recall that $x_2-y_2=x+(b-y)-J_1=b-(J_1+(y-x))>0$. Stop stage at time $\tau_2=$ either
copies meet or when the distance from $Y$ to $X$ is $b-J_1$.\\
\textbf{First case:} The coupling occurs and we are done.
\\
\textbf{Second case:} The initial distance from $Y$ to $X$ is $(b-J_1)-(y-x)$. Therefore, $Y$ moved $\frac{y-x}{2}$ to the left, i.e. $y_3=y_2-\frac{y-x}{2}=J_1-\frac{y-x}{2}>a$ (due to \eqref{key_assump}), and $X$ moves $\frac{y-x}{2}$ to the right, i.e.
$X$ stops at $x_3=x_2+\frac{y-x}{2}=b-\frac{y-x}{2}<b$. Now move to stage 3.
\item[c)]
Let $X=x_3+(B_t-B_{\tau_2})$, $Y=y_3+(B_t-B_{\tau_2})$. Recall that $x_3>y_3$ and $x_3-y_3=b-J_1$. Stop stage c) at time
$\tau_3=$ either $X$ and $Y$ are symmetric or $X$ hits $b$.
\\
\textbf{First case} In the symmetric case we continue the coupling as in Proposition \ref{keyprop}. 
\\
\textbf{Second case:} Here we have $X_{\tau_3-}=b$ and $Y_{\tau_3} = b-(b-J_1)=J_1$. At time $\tau_3$ we let $X$ jump to $J_1$, i.e. we use the same $J_1$ as has earlier been used for $Y$. Hence the coupling occurs and we are done.  
\end{enumerate}
First let us observe that the above construction is in fact a coupling for the processes, i.e. both processes
have the same marginal distributions. 
Denote $\tau_{hit}=\inf\{t:X_t \mbox{ hits the boundary }\partial I\}$. Then, for $t<\tau_{hit}$ we have by construction of $X_t$ 
\[x+B_t,\]
since all changes of sign of the BM appear by definition in the $Y$-process. This immediately implies that $X$, up to its first arrival to the boundary,
is independent of $J_1$ and behaves as a BM. It remains to show that $Y$ also behaves as a BM in the interior of the domain. Now let $\tau_{hit}^1$ ($\tau_{hit}^2$) be the first (second) time when $Y$ hits $\partial(I)$ and observe that by construction for $t<\tau_{hit}^1$, 
$Y$ can be written as
\[Y_t=y+\int_{0}^t f(s)dB_s\] 
and similarly for $\tau_{hit}^2 > t\ge \tau_{hit}^1$
\[Y_t=J_1+\int_{\tau_{hit}^1}^t f(s)dB_s\] 
where $f(s)$ is $\mathcal{G}_s$-measurable and only takes values in $\{-1,1\}$, because $f$ keeps track of the
changes of sign of the BM as defined in the coupling construction. But up to time $t$ these changes only depend on the $\mathcal{F}_t$-history of the 
BM and on $\sigma(J_1)$.  
But from here it is easy to see, e.g. using Levy's characterization of the BM that $Y$ also behaves as a BM in
$I$.
\begin{subsection}{Analysis of the construction}
The crucial observation in the analysis of the coupling construction is that in each stage the stopping rule is determined by an exit time of a standard BM on an interval of length less or equal to $\frac{b-a}{2}$. These exit times (which are not independent) can be dominated by a sum of 
independent exit times of BM's of an interval of length $\frac{b-a}{2}$, as we will see in the sequel. 
\begin{enumerate}
\item[a)]
\textbf{First case:} In this case we stop the process when the two copies are symmetric about the origin. For this to occur the BM has to arrive to $\frac{(a+b)-(y-x)}{2}-x$ (terminal $-$ initial initial location of $X$)$= \frac{(a+b)-(y+x)}{2}<0$.\\
\textbf{Second case:} Here we stop when the coordinate $Y$ hits $b$, i.e. in the case when the BM $B$ has hit $b-y>0$.\\

Thus, stage a) ends when the BM exists an interval of length $b-y+\frac{(y+x)-(a+b)}{2}<\frac{b-a}{2}$.\\
\item[b)]
\textbf{First case:} We stop when $X$ and $Y$ arrive at the point $-\frac{x_2-y_2}{2}=\frac{(y-x)-(b-J_1)}{2}$.\\
\textbf{Second case:} We stop when BM arrives to $\frac{y-x}{2}>0$.\\

Thus, stage b) ends when BM exists an interval of length $\frac{b-J_1}{2}<\frac{b-a}{2}$\\

\item[c)]
\textbf{First case:} We stop when $X$ and $Y$ are in symmetric locations. Initially, $X$ is in $x_3=b-\frac{y-x}{2}$, and $X$ is kept $b-J$ above
$Y$ throughout the stage. Therefore we stop when the BM arrives to $\frac{(a+b)+(b-J)}{2}-x_3=\frac{a+(y-x)-J}{2}<0$.\\
\textbf{Second case:} We stop when BM hits $b-x_3=\frac{y-x}{2}$.\\

Thus, stage c) ends when BM exists an interval of length $\frac{J-a}{2}<\frac{b-a}{2}$.
\end{enumerate}
Now we are able to finish the proof of Theorem \ref{main} by estimating the tails of the coupling time.
\begin{proof}(Theorem \ref{main})
Now let $\xi_1,\xi_2,\ldots,\xi_5$ be independent exit times of standard Brownian motion of the interval $\partial\tilde{I}=(-\frac{a+b}{2},\frac{a+b}{2})$.
Then, by Proposition \ref{firstprop} and Proposition \ref{keyprop} we can conclude that the probability of the event
\begin{itemize}
\item 
both processes bave been coupled before time $t$ \textbf{or} we have moved to stage b)  before $t$
\end{itemize} 
is dominated by $\mathbb{P}\bigl(\xi_1+\xi_2>t\bigr)$. Continuing this line of reasoning we see that the coupling time $\tau_{coupl}$ is dominated by
$\sum_{i=1}^{5}\xi_i$ in the sense that 
\begin{equation*}\label{e:comparison}
\mathbb{P}(\tau_{coupl}>t)\le\mathbb{P}\biggl(\sum_{i=1}^{5}\xi_i>t\biggr).
\end{equation*}
Now we obtain  
\begin{equation*}  
\begin{split}
\mathbb{P}\bigl(\tau_{coupl}>t\bigr)&\le \mathbb{P}\biggl(\sum_{i=1}^{5}\xi_i>t\biggr) \leq e^{-\lambda t}\mathbb{E}\biggl[e^{\lambda \sum_{i=1}^{5}\xi_i}\biggr]\\
&= \left(\mathbb{E}\bigl[e^{\lambda\xi_1}\bigr]\right)^5e^{-\lambda t}.
\end{split}
\end{equation*}
Since $\mathbb{E}\bigl[e^{\lambda\xi_1}\bigr]<\infty$ for all $\lambda<\lambda_1^{(a,b)}$, the claim follows.
\end{proof}

\end{subsection}

\begin{remark}
For pure probabilists the application of the more analytic Theorem \ref{contofgap} via Corollary \ref{c:redcompsupp} may be somewhat unsatisfactorily; we included this result due to its applicability in the multi-dimensional setting. It seems to be very probable that Theorem \ref{contofgap} may be replaced by an additional coupling argument in the proof of Theorem \ref{main}. 
\end{remark}

\section{Different jump distributions}
We end this work with some remarks concerning the case of a one-dimensional BM in $(a,b)$ with two different jump distributions $\nu_a$ and $\nu_b$. It is a natural question, whether one can adopt the methods of this paper in order to prove another very recent results of Li, Leung and Rakesh \cite{LLR} and Li and Leung \cite{LL}, respectively, namely
\begin{equation}\label{sup}
 \sup_{\nu_a,\nu_b} \gamma_1(\nu_a,\nu_b) = \lambda_2^D
\end{equation}
and
\begin{equation}\label{inf}
\inf_{\nu_a,\nu_b}\gamma_1(\nu_a,\nu_b) = \lambda_0^D
\end{equation}
It is known that the supremum is \eqref{sup} is attained at $\nu_a=\delta_{((a+2b)/3}$ and $\nu_b = \delta_{(2a+b)/3}$ and that the infimum in \eqref{inf} is never attained. Let us remark that equation \eqref{inf} is from a heuristic point of view probabilistically rather clear, as taking $\nu_{a,n}:=\delta_{a+1/n}$ and $\nu_{b,n}:=\delta_{b+1/n}$ one might expect that as $n \rightarrow \infty$ one gets a BM in $(a,b)$ with reflecting boundary conditions.

At the moment we can extend our methods only to certain very special classes of of jump distributions $\nu_a$ and $\nu_b$, but these considerations indicate that with more effort one might be able to extend our methods to the more general situation of different jump distributions. 
Still at the moment we are far from being able to present a probabilistic proof of the beautiful assertions \eqref{sup} and \eqref{inf}, but we believe that a solution of the following problem deepens the probabilistic understanding of the large time behavior of the BM with jump boundary considerably.
\\
\\
\textbf{Open Problem:} Find a coupling approach to the recent results \eqref{sup} and \eqref{inf} of Leung and Li.
\\
\\
We hope to come back to this problem in a subsequent publication. 
\section*{Acknowledgements}
The authors would like to thank Heinrich von Weizs\"acker (Kaiserslautern) for initiating this collaboration by establishing the first contact between the authors and Ross Pinsky (Haifa) for several useful comments concerning the topic of this work as well as the probability theory group at the LMU Munich for their financal support of this collaboration. Moreover, we would like to thank an anonymous referee, whose detailed report improved the paper considerabely.


\begin{thebibliography}{99}
\bibitem{Ar} D. G. Aronson, \textit{Bounds on the fundamental solution of a parabolic equation}, Bull. Amer. Math. Soc. 73 (1967), 890--896
\bibitem{As}  S\"oren Asmussen, Applied probability and queues. Second edition. Applications of Mathematics (New York), 51. Stochastic Modelling and Applied Probability. Springer-Verlag, New York, 2003
\bibitem{D} E. Brian Davies, Spectral Theory and Differential Operators, Cambridge studies in advanced mathematics, Cambridge University Press, Cambridge, 1996
\bibitem{LL} Wenbo V. Li and Yuk J. Leung, \textit{Fastest rate of convergence for Brownian motion with jump boundary}, preprint
\bibitem{LLR} Wenbo V. Li, Yuk J. Leung and Rakesh, \textit{Spectral analysis of Brownian motion with jump boundary},  Proceedings of the American Mathematical Society. Vol 136, 4427-4436, (2008).
\bibitem{BaP1} Iddo Ben-Ari and Ross G. Pinsky, \textit{Spectral analysis of a family of second-order elliptic operators with nonlocal boundary condition indexed by a probability measure},  J. Funct. Anal.  251  (2007),  no. 1, 122--140. 
\bibitem{BaP2} Iddo Ben-Ari and Ross G. Pinsky, \textit{Ergodic behavior of diffusions with random jumps from the boundary}, 
Stochastic Process. Appl. 119 (2009), no. 3, 864--881. 
\bibitem{EN} Klaus-Jochen Engel and Rainer Nagel: One-parameter semigroups for linear evolution equations. Graduate Texts in Mathematics, 194. Springer-Verlag, New York, 2000. 
\bibitem{GQZ} Guang Lu Gong, Min Ping Qian and Zhong Xin Zhao, \textit{Killed diffusions and their conditioning}, Probab. Theory Related Fields 80 (1988), 151--167. 
\bibitem{GK1} Ilie Grigorescu and Min Kang, \textit{Brownian motion on the figure eight}, J. Theoret. Probab.  15  (2002),  no. 3, 817--844.
\bibitem{GK2} Ilie Grigorescu and Min Kang, \textit{Ergodic properties of multidimensional Brownian motion with rebirth}.  Electron. J. Probab.  12  (2007), no. 48, 1299--1322
\bibitem{GK3} Ilie Grigorescu and Min Kang, \textit{The Doeblin condition for a class of diffusions with jumps}, preprint (2009)
\bibitem{Kato} Tosio Kato, Perturbation theory for linear operators. Reprint of the 1980 edition. Classics in Mathematics. Springer-Verlag, Berlin, 1995. xxii+619 pp.
\bibitem{K} Elena Kosygina, \textit{Brownian flow on a finite interval with jump boundary conditions}, Discrete Contin. Dyn. Syst. Ser. B  6  (2006), 867--880
\bibitem{KM} Martin Kolb and Achim W\"ubker, \textit{Spectral Analysis of Diffusions with Jump Boundary}, submitted
\bibitem{P} Ross G. Pinsky, \textit{On the convergence of diffusion processes conditioned to remain in a bounded region for large time to limiting positive recurrent diffusion processes}, Ann. Probab. 13 (1985), 363--378. 
\bibitem{P95} Ross G. Pinsky, Positive Harmonic Functions and Diffusion, Cambridge University Press, 1995
\bibitem{SVII} Daniel W. Stroock and S.R.S Varadhan, \textit{Diffusion Processes with Continuous Coefficients, II}, Comm. Pure Appl. Math. 22 (1967), 479--530
\bibitem{T} Jozef L. Teugels, \textit{On the Rate of Convergence in Renewal and Markov Renewal Processes}, 1967
\end{thebibliography}
\end{document}